\documentclass[preprint,12pt]{elsarticle}

\usepackage{graphicx}
\usepackage{amssymb}
\usepackage{lineno}
\usepackage{amsthm}
\usepackage{marvosym}
\usepackage{amsmath}
\usepackage{color}
\usepackage{csquotes}

\usepackage{algorithmic}
\usepackage{algorithm}

\newcommand{\algorithmicbreak}{\textbf{break}}
\newcommand{\BREAK}{\STATE \algorithmicbreak}
\newcommand{\algorithmiccontinue}{\textbf{continue}}
\newcommand{\CONTINUE}{\STATE \algorithmiccontinue}

\newtheorem{theorem}{Theorem}[section]
\newtheorem{remark}{Remark}[section]

\makeatletter
\def\ps@pprintTitle{%
   \let\@oddhead\@empty
   \let\@evenhead\@empty
   \let\@oddfoot\@empty
   \let\@evenfoot\@oddfoot
}
\makeatother

\begin{document}
	\begin{frontmatter}	
			\title{  Rectangularly Dualizable Graphs: Area-Universality }

	\author{Vinod Kumar\corref{cor1}\fnref{label2}}
	\ead{vinodchahar04@gmail.com}
	\author{	Krishnendra Shekhawat\fnref{label2}}
	\ead{krishnendra.shekhawat@pilani.bits-pilani.ac.in}
	\fntext[label2]{Department of Mathematics, BITS Pilani, Pilani Campus, Rajasthan-333031, India }
	\cortext[cor1]{Corresponding Author}


\begin{abstract}
A plane graph is called a {\it rectangular graph} if  each of its edges  can be oriented  either horizontally or vertically, each of its interior regions is  a four-sided region and all interior regions  can be fitted in a rectangular enclosure.  If the dual of a plane  graph   is a   {\it rectangular  graph}, then  the plane graph   is a  {\it rectangularly dualizable graph}.  A rectangular dual  is {\it area-universal} if any assignment of areas to each of its regions can be realized by a combinatorially weak equivalent  rectangular dual.  It is still unknown   that there exists no  polynomial time  algorithm to construct an area-universal rectangular dual for  a {\it rectangularly dualizable graph }. In this paper, we describe a class  of  {\it rectangularly dualizable graphs} wherein each graph can be realized by an area-universal rectangular dual. We also present  a  polynomial time algorithm for  its construction.
\end{abstract}

\end{frontmatter}

\section{Introduction}
\label{sec1}
A plane graph is called a {\it rectangular graph} if  each of its edges  can be oriented  either horizontally or vertically, each of its interior regions is  a four-sided region and all interior regions  can be fitted in a rectangular enclosure. A graph $\mathcal{H}$ is  dual of a  plane graph $\mathcal{G}$ if vertices of $\mathcal{G}$  correspond to regions of $\mathcal{H}$  and whenever two  vertices are adjacent, the corresponding regions of  $\mathcal{H}$ are adjacent.  If the dual of a plane  graph   is a   {\it rectangular  graph}, then  the plane graph   is a  {\it rectangularly dualizable graph} (RDG). The only missing information in the dual of an RDG to be a rectangular dual is  to assign  horizontal or vertical orientation to  each of its edges.  The point where three or more rectangles of a given rectangular dual meet is called a joint.  A rectangular dual has 3-joints and 4-joints only where 4-joints are regarded as a limiting case of 3-joints \cite{rinsma1988existence}. Hence abiding by the common design practice, we consider rectangular duals with 3-joints in this paper. Thus, a rectangular dual of a plane graph can be seen as a partition of a rectangle into $n-$ rectangles provided that no four of them meet at a point. 
\par
A rectangular dual is {\it area-universal} \cite{Eppstein12} if any assignment of areas to each of its rectangles can be realized by a  {\it combinatorially weak equivalent} rectangular dual.  If there associate areas to each of its rectangles, then rectangular dual  is  known as  a {\it rectangular cartogram}.  More precisely, a rectangular cartogram  is a map in which  each region is a rectangle enclosing a rectangular area where the size of each rectangle  relies on the area assigned to it. Without knowing the concept of area-universality, it is difficult  to find a desired rectangular cartogram. 
 \par
The rectangular cartograms   \cite{raisz1934rectangular}  are used to visualize spatial information (it may be economic strength, population etc.) of geographic regions, i.e. they are used to display more than one quantity associated with the same set of geographic regions (in  \cite{raisz1934rectangular}, the population, land area and wealth within the United States were shown as cartograms). The visual comparison of multiple cartograms corresponding to the same set of geographic regions can be made easier if each of the cartograms is area-universal.  
\par
A VLSI system structure is described by a graph where vertices correspond to component modules and edges correspond to required interconnections. Area-universal rectangular duals are useful in controlling aspect ratios  \cite{Young01,Hong04,Murata98}. At an early stage, the areas of soft modules are not yet known, however the relative positions of modules of a VLSI circuit are known. Constructing an area-universal rectangular dual enables us to assign areas  to its rectangles (modules) at the later stage. Thus, the ability of constructing an area-universal rectangular dual at an early stage optimizes circuit's area at the later stage.

 \subsection{\textbf{Related work}} 
Only plane graph can be dualized. The class of RDGs  is very restrictive \cite{Kozminski85,Lai90}. Not every RDG can  be realized by an area-universal rectangular dual \cite{Rinsma87,Eppstein12}.  Rinsma \cite{Rinsma87} described a vertex-weighted outer planar $\mathcal{G}$ (area is assigned to each of its  vertex)  such that there exists no rectangular dual for $\mathcal{G}$ having these weights as  rectangles' areas. Thus it is interesting to know when a  rectangularly dualizable graph can be realized by  an area-universal rectangular dual. Recently, Eppstein et al. \cite{Eppstein12} derived the following necessary and sufficient condition for a rectangular dual to be area-universal. 
\begin{theorem}
{\rm \cite[Theorem 2]{Eppstein12} A rectangular dual $F$ is area-universal if and only if every maximal internal line segment of $F$ is a side of atleast one of its component rectangles.}  
\end{theorem}
A line segment in a rectangular dual $F$ is formed by the consecutive inner edges of $F$. A maximal line segment is the one which is not contained in any other line segments. The rectangular dual shown in Fig. \ref{fig:f1}a is area-universal since every maximal internal line segment is the side of  its rectangles (for example, the bold line  segment $p$ in Fig. \ref{fig:f1}a is the side of a rectangle) whereas the rectangular dual shown in Fig. \ref{fig:f1}b is not area-universal since the  maximal   line segment $s$ (the bold line segment) is not the side of any of its rectangles. 
\begin{figure}[H]
	\centering
	\includegraphics[width=0.6\linewidth]{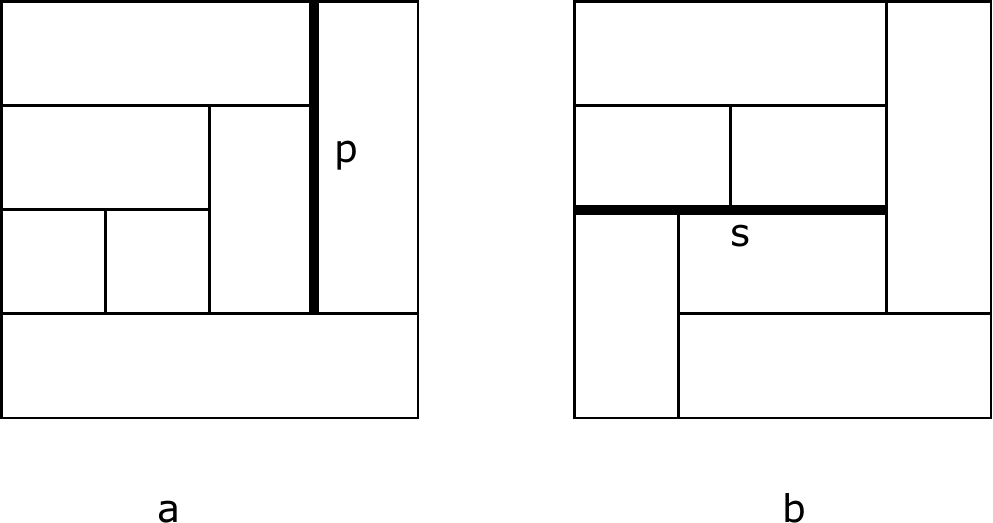}
	\caption{(a) An area-universal rectangular dual and (b) an rectangular dual that is not area-universal. }
	\label{fig:f1}
\end{figure}
Eppstein \cite{Eppstein12} et al. described an algorithm  that constructs an area-universal rectangular dual for an RDG $\mathcal{G}$, if it exists. The computational complexity of this algorithm is O$(2^{\text{O}(K^2)}n^{\text{O}(1)})$ where $K$ is the maximum number of 4 degree vertices in any minimal separation component. This algorithm is not fully polynomial. For instance, if $K$ is fixed, it runs in a polynomial time but in general, it runs in an exponential time. 
\par
The heuristic algorithms for constructing area-universal rectangular duals can be seen in \cite{wimer1988floorplans,Marc07}. The area-universal convex polygonal drawings for biconnected outer planar graphs are given in \cite{chang2017area}.  Recently, it has been shown that a planar graph $G$ is area-universal if for any area assignment  to the inner regions of  $G$, a straight line drawing of $G$ can be realized \cite{kleist2018drawing,11303_8540} and the area-universality  of  subgraphs of such graphs can be seen in \cite{biedl2013drawing}. The extension of the concept of area-universality to rectilinear duals has been given in \cite{Alam13,kawaguchi2007orthogonal}.
 \subsection{\textbf{Results}} 
 Any efficient algorithm to construct an area-universal rectangular dual for a given RDG, if it exists,  still unknown. In this paper,  we describe a class of RDGs in which each RDG can be realized by  an area-universal rectangular dual  in polynomial time.   We also show that every  induced rectangular dualizable subgraph of this class can also be realized by  an area-universal rectangular dual. This work can be seen in  Section \ref{sec2} and its conclusion  is given in Section \ref{sec3}.  
 \par
 Table \ref{tab1} depicts a list of notations used in this paper.
  
  \begin{table}[H] 
  	\centering
  	\begin{tabular}{|p{2cm}|p{8.5cm}|}
  		\hline
  		Symbol & Description \\	
  		
  		\hline
  		$\mathcal{G}$ &  rectangularly dualizable graph \\
  		\hline
  		
  		RDG(s) &   rectangularly dualizable graph(s)\\
  		\hline
  		$v_i$ & $i^{\text{th}}$ vertex of a graph\\	
  		\hline
  		$d(v_i)$ & degree of  $v_i$\\	
  		\hline
  		
  		$N(v_i)$ & neighborhood of $v_i$\\	
  		\hline
  			$|P|$ &  cardinality of $P$\\	
  		\hline
  		$R_i$ & $i^{th}$ component rectangle of a rectangular dual corresponding to $v_i$\\	
  		\hline
  	\end{tabular} 
  	
  	\caption{\rm List of Notations} 
  	\label{tab1} 
  	
  \end{table}


\section{Area-Universal Rectangular Dual Construction }\label{sec2}
In this section, we first identify a class $\mathcal{C}$ of  (RDGs) in which  every graph admits an area universal rectangular dual. Then we present four   algorithms.   Algorithm \ref{algo21} computes a set $P$ of 4 degree vertices forming a path in a given RDG $\mathcal{G}$ which is the input for Algorithm \ref{algo22} as well as Algorithm \ref{algo23}. Algorithm \ref{algo22} decides whether the given RDG belongs to  $\mathcal{C}$. Algorithm \ref{algo23} constructs  an area-universal rectangular dual for the given RDG that belongs to $\mathcal{C}$. Algorithm \ref{algo24} computes the position of rectangles to be inserted in the rectangular dual. It is used as a call function by Algorithm \ref{algo23}. Then  we analyze the computation complexity of these algorithms. The proof of correctness of  Algorithm \ref{algo23} follows from Theorem \ref{thm22} and the characterization of $\mathcal{C}$ follows from  Theorem \ref{thm23}. It is noted that Algorithm \ref{algo22}  works for biconnected graphs only which may not be seen as a restriction since a separable connected graph can be partitioned into  biconnected subgraphs.
\par
The class $\mathcal{C}$ of RDGs is defined as follows: Let $\mathcal{G}$ be an RDG. Let $L_1$ be a set of 4 degree vertices forming a path in $\mathcal{G}$ such that $|L_1|>1$. Let $L_i$ $(i\geq 2)$ be the sets of vertices of $\mathcal{G}$ such that $L_{i-1}$ $\subset L_i$ and has exactly one more vertex $v_s$ other than the vertices of  $L_{i-1}$ with the property $|N(v_s) - (N(v_s) \cap L_{i-1})| \leq 3.$

\begin{algorithm}[H]	
	\caption{ \textbf{getDegree$4$AdjacentVertices}} 
	\begin{algorithmic}[1]

		\REQUIRE  An RDG $\mathcal{G}=(V,E)$ 
		\ENSURE	 A set $P$ of 4-degree vertices forming a path
		\STATE  $P = \phi$
		\STATE 	checked = $\phi$
		\STATE chainStarted $=$ true
		\FORALL {$v \in V$ }
		\IF {$d(v)==4$}
		\WHILE {chainStarted $==$ true}
		\STATE $P=P\cup \{v\}$
		\STATE checked$=$checked $\cup \{v\}$
		\FOR {$u\in (N(v)-$ checked)}
		\STATE chainStarted $=$ true
		\IF { $d(u)==4$}
		\STATE $v=u$
		\BREAK
		\ENDIF
		\ENDFOR
		\STATE chainStarted $=$ false
		\ENDWHILE
		\IF {$|P|>1$}
		\RETURN $P$
		\ELSE
		\PRINT   $\mathcal{G}$ does not belongs to $\mathcal{C}$.
		\ENDIF
		\ENDIF  
		\ENDFOR	
	\end{algorithmic} \label{algo21}
	
\end{algorithm}

As the elements of $P$ forms a path, therefore they can be indexed by another set $P_1=\{v_1, v_2, \dots, v_k\}$ in such a way that $v_i$ is adjacent to $v_{i+1}$ where $v_1$ is the initial point and $v_k$ is the end point of the path. We use this indexed set $P_1$ as the input requirement of Algorithm \ref{algo23}.
\par

\begin{algorithm}[H] 	
	\caption{\textbf{Checking whether $\mathcal{G}$ belongs to $\mathcal{C}$}} 
	\begin{algorithmic}[1]
		\REQUIRE An RDG $\mathcal{G}=(V,E)$ and $P$ (computed by Algorithm \ref{algo21}),
		
		\ENSURE	Either $\mathcal{G}$ belongs to $\mathcal{C}$ or not
		\STATE $L =P$
		\FORALL	{$v_i \in V-L$}
		\IF    {$v_i$ is adjacent to a vertex of $L$ such that $|N(v_i) - (N(v_i) \cap L)| \leq 3$}
		\STATE	$L =L \cup \{v_i\}$ 
		\ELSE
		\CONTINUE
		\ENDIF
		\ENDFOR
		
		\IF {$L = =V$}
		\RETURN $\mathcal{G}$ belongs to $\mathcal{C}$
	    \ELSE
		\PRINT  $\mathcal{G}$ does not belong to $\mathcal{C}$
		\ENDIF
	\end{algorithmic} \label{algo22}
\end{algorithm}

\par  
For an example, consider an RDG  $\mathcal{G}_1$ shown in Fig. \ref{fig:f2}. For $\mathcal{G}_1$, we have 

$L_1=\{v_2,v_3\}$,

$L_2=\{v_1,v_2,v_3\}$ and  $|N(v_1) - (N(v_1) \cap L_1)|=2$,

$L_3=\{v_1,v_2,v_3, v_4\}$ and  $|N(v_4) - (N(v_4) \cap L_2)|=2$, 

$L_4=\{v_1,v_2,v_3, v_4,v_5\}$ and  $|N(v_5) - (N(v_5) \cap L_3)|=3$,

$L_5=\{v_1,v_2,v_3,v_4,v_5, v_6\}$ and  $|N(v_6) - (N(v_6) \cap L_4)|=2$,

$L_6=\{v_1,v_2,v_3,v_4,v_5, v_6,v_7\}$ and  $|N(v_7) - (N(v_7) \cap L_5)|=3$,

$L_7=\{v_1,v_2,v_3,v_4,v_5, v_6,v_7,v_8\}$ and  $|N(v_8) - (N(v_8) \cap L_6)|=1$,

$L_8=\{v_1,v_2,v_3,v_4,v_5, v_6,v_7,v_8,v_9\}$ and  $|N(v_9) - (N(v_9) \cap L_7)|=1$,

$L_9=\{v_1,v_2,v_3,v_4,v_5, v_6,v_7,v_8,v_9,v_{10}\}$ and  $|N(v_{10}) - (N(v_{10}) \cap L_8)|=0$.

This shows that $\mathcal{G}_1$ in Fig. \ref{fig:f2} belongs to $\mathcal{C}$.
Also, for $\mathcal{G}_1$, 
Algorithm \ref{algo21} computes the set $P=\{v_2,v_3\}$ of 4 degree vertices, which is the only set of 4 degree vertices.
But there can be more than one path  consisting of 4 degree vertices in $\mathcal{G}$. In such case,  using Algorithm \ref{algo22}, we first check whether $\mathcal{G}$   belongs to $\mathcal{C}$ for each such set. If we have more than one  such set for which $\mathcal{G}$ belongs to $\mathcal{C}$, we can derive  an area-universal rectangular dual for each such set. If  for every such set, $\mathcal{G}$ fails to belong to $\mathcal{C}$, then  it is inconclusive whether $\mathcal{G}$ admits an area-universal rectangular dual. Another case for inconclusiveness is that each set is either empty or singleton (line 20, Algorithm \ref{algo21}).

Let $F$ denote a rectangular dual and $R_i((x_i,y_i),l_i,h_i)$ denote the $i^\text{th}$ component rectangle corresponding to  $v_i$ with bottom left coordinate $(x_i,y_i)$, length $l_i$ and height $h_i$.
\par

In Algorithm \ref{algo23},  the input graph is same as that of Algorithm \ref{algo22}.

\begin{algorithm}[H] 
	\caption{ \textbf{Constructing an area-universal rectangular dual}} 
	\begin{algorithmic}[1]
		\REQUIRE   $\mathcal{G}=(V,E)$,  $P_1$,  $N(v_i)$,  $\forall v_i \in V$, $k$ and $l$ are  integers such that $k=|P_1|$ and $l=|N(v_t) \cap L|$, a fix point $(x_0, y_0)$ in the plane
		\ENSURE An area-universal rectangular dual of $\mathcal{G}$ 
		\STATE $L = \phi$	
		\STATE $F = \phi$
		\FORALL {$i$ from $i=1$ to $k$}
		\FORALL { $v_i \in P_1$}
		\STATE {$x_0=x_i, y_0=y_i$}
		\STATE 	InsertRectangle $R_i((x_i,y_i), 1,1)$ in $F$
		\STATE  $x_0 \leftarrow x_0+1$
		\STATE $L = L \cup \{v_i\}$
		\ENDFOR
		\ENDFOR
		\FORALL {$v_t \in (V-L)$} 
		\IF {$|N(v_t) - (N(v_t) \cap L)| \leq 3 $} 
		\IF {$N(v_t) \cap L = =\{v_k, 1\leq k\leq l \}$}
		\STATE InsertRectangleFunction$(F, v_t, L, N(v_t))$
		\STATE $L = L \cup \{v_t\}$
		\ENDIF
		\ENDIF
		\ENDFOR
		\RETURN  $F$
\end{algorithmic}	\label{algo23}
\end{algorithm}

\begin{figure}[H]
	\centering
	\includegraphics[width=.52\linewidth]{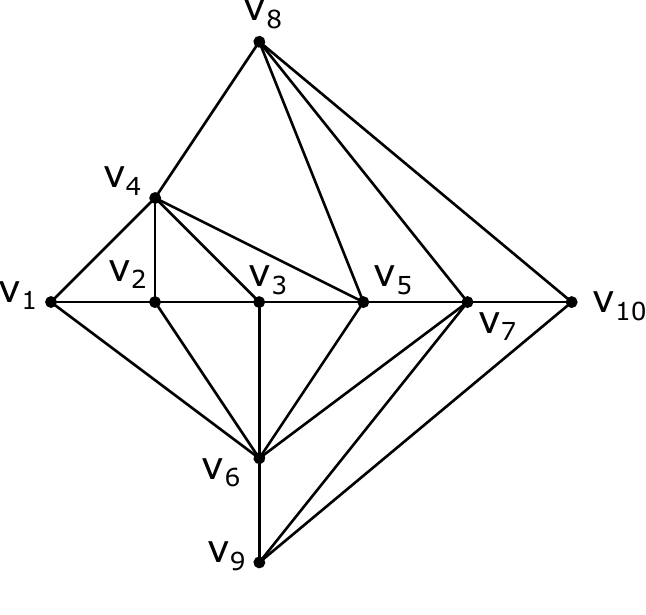}
	\caption{\rm An RDG $\mathcal{G}_1$ that belongs to $\mathcal{C}$.}
	\label{fig:f2}
\end{figure}

\begin{algorithm}[H]
	\caption{ \textbf{InsertRectangleFunction$(F, v_t, L, N(v_t))$}} 
	\begin{algorithmic}[1]
		\REQUIRE $F, v_t, L, N(v_t)$ 
		\ENSURE	 $F$ is updated with the addition of rectangle $R_t$ for some vertex $v_t$.
		\IF { $y_1==y_2==...==y_l $}
		\STATE $x_p=\text{min}\{x_k, 1\leq k \leq l\} $ 
		\STATE insertRectangle $R_t((x_p,y_p-1), \sum\limits_{i=1}^{l}l_i,1)$ in $F$ 
		 
		\ELSIF { $x_1==x_2==...==x_l $}
		\STATE $y_p=\text{min}\{y_k, 1\leq k \leq l\} $ 
		\STATE insertRectangle $R_t((x_p-1,y_p), 1,\sum\limits_{i=1}^{l}h_i)$ in $F$
		
		\ELSIF { $\text{there exists}~x_i \neq x_j, i \neq j, 1\leq i,j \leq l$}
		\STATE $y_p=\text{min}\{y_k, 1\leq k \leq l\} $ 
		\STATE insertRectangle $R_t((x_p+l_p,y_p), 1,\sum\limits_{i=1}^{l}h_i)$ in $F$
		
		\ELSE
		\STATE $x_p=\text{min}\{x_k, 1\leq k \leq l\} $ 
		\STATE insertRectangle $R_t((x_p,y_p+h_p), \sum\limits_{i=1}^{l}l_i,1)$ in $F$
		\ENDIF	
	\end{algorithmic} \label{algo24}
	
\end{algorithm}

Now we construct an area-universal rectangular dual for the graph $\mathcal{G}_1$ in Fig. \ref{fig:f2} using Algorithm \ref{algo23}.  Algorithm \ref{algo23} first updates $F$ by inserting  rectangles  $R_2$ and $R_3$  dual to the vertices of $P$ as shown in  Fig. \ref{fig:f3}a. Again, Algorithm \ref{algo23} recursively updates $F$ by inserting  rectangles   dual to the vertices of  $V-P=\{v_1,v_4,v_5,v_6, v_7,v_8,v_9,v_{10}\}$ such that $|N(v_i) - (N(v_i) \cap L)|\leq 3$ as shown in Fig. \ref{fig:f3}b-\ref{fig:f3}i.  Thus, Algorithm \ref{algo23} constructs an area-universal rectangular dual as shown in Fig. \ref{fig:f3}i  for $\mathcal{G}_1$.

\begin{figure}[H]
	\centering
	\includegraphics[width=.73\linewidth]{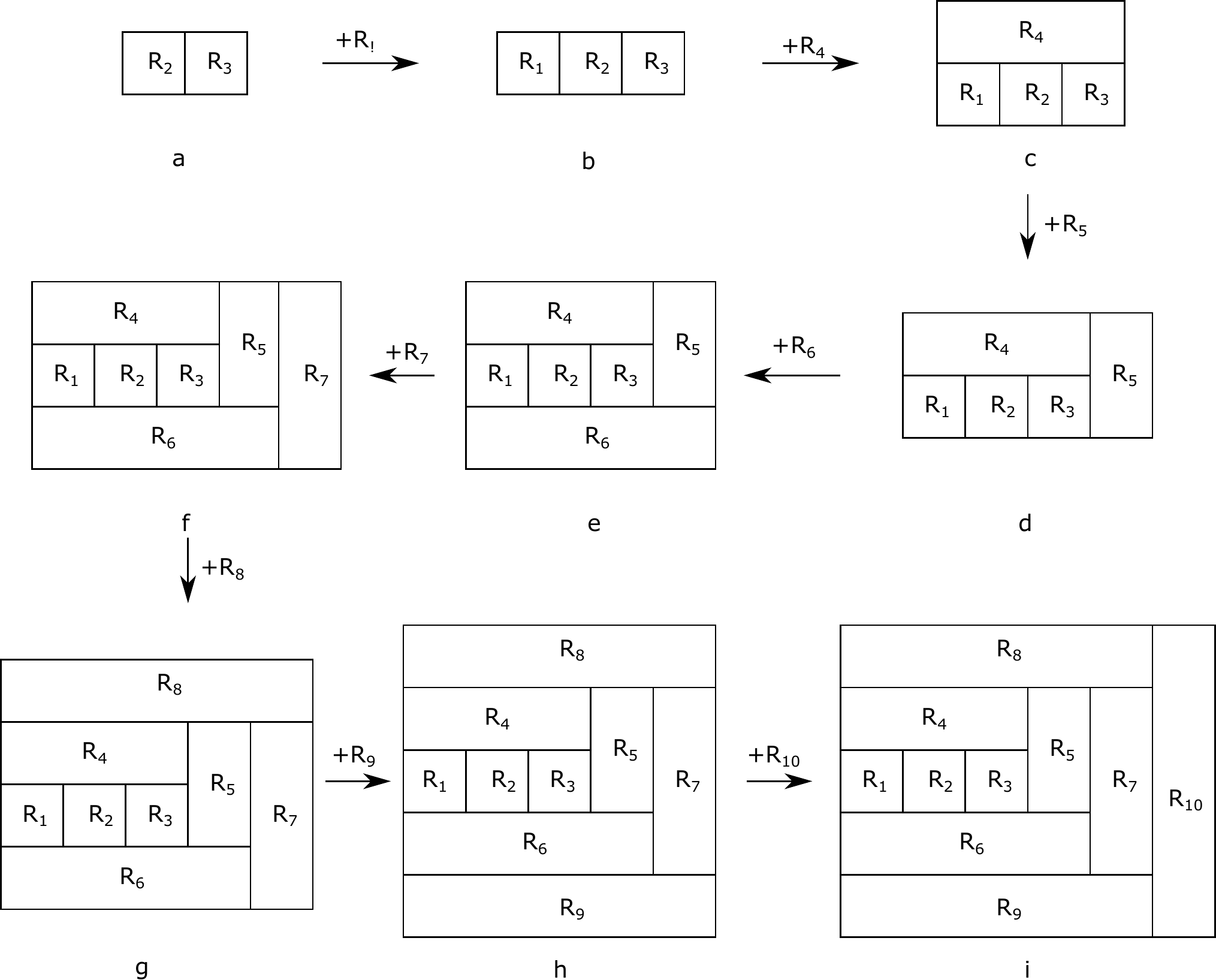}
	\caption{Constructing an area-universal rectangular dual for $\mathcal{G}_1$ in Fig. 2.}
	\label{fig:f3}
\end{figure}

\subsection*{\textbf{ Analysis of computational complexity}}

For Algorithm \ref{algo21}, to search $P$, the entire list of vertices requires  scanning in worst case. This implies that the computational complexity of Algorithm \ref{algo21} is $|P||N(v_s)| \cong$ O$(n)$ where $v_s$ is a vertex of the largest degree.
\par
For Algorithm \ref{algo22}, it is easy to note that  $|V-L|$ is large if and only if $|L|$ is small. This implies that both of them can not  approach to $|V|$ simultaneously. Therefore,  the computational complexity of Algorithm \ref{algo22} is  $|V-L||(N(v_i) \cap L)|$ = $|N(v_i)||L||V-L|\cong|N(v_s)||L||V|$ $\cong$ O$(n)$ where $v_s$ is the vertex of the largest degree. Hence, the computational complexity of Algorithm \ref{algo22} is linear. 
\par
The computational complexity of  lines  $3-9$ and lines $10-17$  in Algorithm \ref{algo23} are $k|P_1|$  and $|V-L|.|(N(v_i) \cap L)|$ respectively.   The complexity of $|(N(v_i) \cap L)|$ is $|N(v_s)| \times |L|$ where $v_s$ is the vertex of the largest degree.  Now $\text{max} \{k|P_1|, |V-L|.|(N(v_s)\cap L)|\} \cong |V|.|N(v_s)||L|\cong$ O$(n)$. Hence the computational complexity of Algorithm \ref{algo23} is O$(n)$. 

\begin{remark}
{\rm The case when  $|N(v_s)|$ is so large that it approaches to $|V|$ for some graphs, then  the computational complexity of  Algorithm \ref{algo21}, Algorithm \ref{algo22} and Algorithm \ref{algo23} is  O$(n^2)$. However in design problems, such graphs do not appear quite often.
Using binary search method,  this complexity can be reduced to $n\text{log}n$ but at present, we are not considering this method.} 
\end{remark}

\begin{figure}[H]
	\centering
	\includegraphics[width=0.95\linewidth]{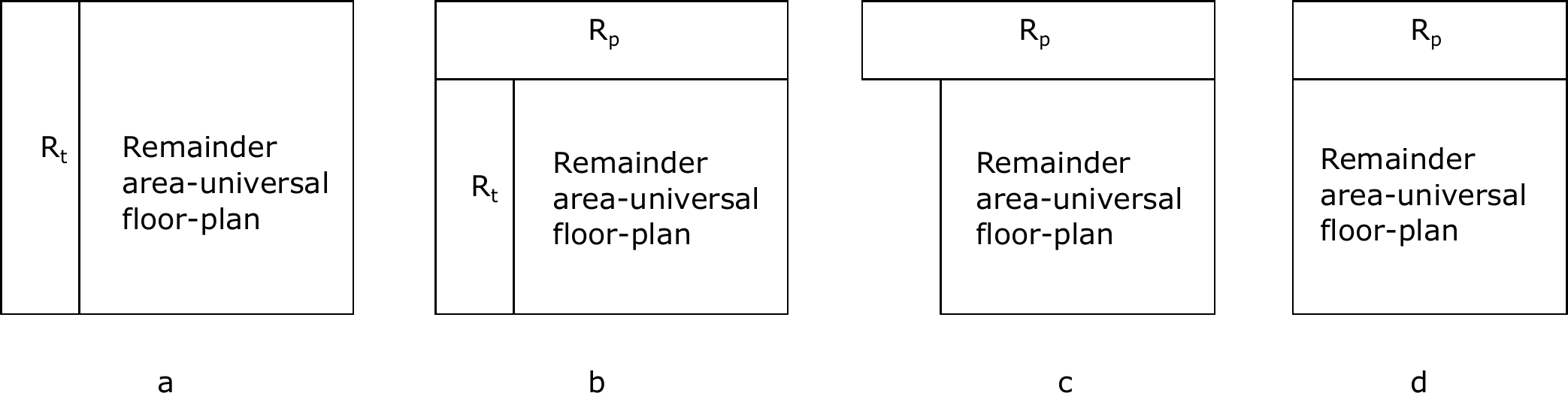}
	\caption{(a-b) Possible positions for an exterior rectangle $R_t$ in the rectangular dual $F$ of $\mathcal{G}$, (c) dual $F'$ after deleting $R_t$ from $F$ and (d) adjusting $R_p$ so that the resultant dual  is a rectangular dual for $H_1$.}
	\label{fig:f4}
\end{figure}

\begin{theorem} \label{thm22}
{\rm For  every  RDG $\mathcal{G}$ of the class $\mathcal{C}$, Algorithm \ref{algo23} constructs an area-universal rectangular dual.} 
\end{theorem}
\begin{proof}
We know that a rectangular dual is area-universal if and only if	 every internal line-segment is maximal. It is noted that Algorithm \ref{algo23} updates an empty dual $F$ by inserting rectangles one by one for each vertex of $\mathcal{G}$ such that every resultant dual  is an rectangular dual and there generates a new internal line segment (one of sides of a rectangle)  while inserting each rectangle (except the first one) in $F$. We prove that every generated internal line segment is maximal and is the side of atleast one of the rectangles of the rectangular dual for $\mathcal{G}$, i.e., we need to prove that each internal line segment is not extendable at its endpoints.
\par
 As $F$ is updated by inserting a rectangle $R_i$ (except for the first one)  by Algorithm  \ref{algo23}, a new  internal line segment $l_i$  is generated in $F$. We claim that $l_i$ does not make an angle of $180^\circ$ with any of the preceding internal line segments (segment). There  arise  the following two possibilities:
 \begin{enumerate}
 \item[i.] $l_i$ is perpendicular to some preceding internal line segments (segment)  meeting at one of the endpoints of these preceding internal line segments (segment) (for instance, the new generated line segment $l_6$ (say) due to the insertion of $R_6$ in Fig. \ref{fig:f3}e is perpendicular to the preceding internal line segments  lying at the  common borders of $R_1$ and $R_2$, $R_2$ and $R_3$, and $R_3$ and $R_5$),
 \item[ii.] none of the endpoints of any of the preceding internal line segments (segment) is intersected by $l_i$ (for instance, the new generated line segment $l_3$ (say) due to the insertion of $R_3$ in Fig. \ref{fig:f3}b does not meet to the preceding internal line segment $l_2$).
 \end{enumerate} 
From above two possibilities, it is clear that  $l_i$ does not allow to extend the preceding internal line segments (segment).  Since $l_i$ is arbitrary. Hence every internal line segment of $F$ is not extendable and is one of the side of rectangles of $F$. Hence the proof.
\end{proof}

\begin{theorem} \label{thm23}
{\rm  Every induced rectangularly dualizable subgraph of each of the graphs of $\mathcal{C}$ admits an area-universal rectangular dual.} 
\end{theorem}
\begin{proof}
Suppose that $\mathcal{H}_k$ is an induced rectangular dualizable subgraph of any of the graph $\mathcal{G} =(V,E)$ of $\mathcal{C}$.   Since  $\mathcal{H}_k$ is a rectangular dualizable graph, it admits a rectangular dual  where no four component rectangles meet at a point.   This implies that the interior regions of $\mathcal{H}_k$ are triangular. Since  $\mathcal{H}_k$ is an  induced subgraph of  $\mathcal{G}$, $\mathcal{H}_k$ is the graph whose vertex set  $V'\subsetneq V$ and the edge set $E'$  contains the edges of  $\mathcal{G}$ whose endpoints are in $V'$. Suppose that $V'=V-S$. Clearly, $S$ is nonempty.  First we claim that $S$ contains an exterior vertex of $\mathcal{G}$. To the contrary, suppose that $S$ has no exterior vertex of $\mathcal{G}$. This implies that  $S$ only contains  the interior vertices (vertex) of $\mathcal{G}$ and hence  $\mathcal{H}_k$ has atleast one region of  length atleast 4 which is a contradiction to the fact that  every interior region of $\mathcal{H}_k$ is triangular.   Therefore, $S$ contains atleast one exterior vertex of $\mathcal{G}$. Let $v_t \in S$ be an exterior vertex of $\mathcal{G}$. Consider an  induced graph $H_1$ of $\mathcal{G}$ obtained by deleting $v_t$ together with all the edges incident to $v_t$.   There arise two possibilities for  deletion of  the rectangle $R_t$ (corresponding to $v_t$) in the rectangular dual $F$ for $\mathcal{G}$ as shown in Fig. \ref{fig:f4}a and \ref{fig:f4}b. Note that the remaining possibilities can be covered by flipping or rotating $F$. The first possibility shows that on deleting $R_t$ from $F$, the remaining part is still an area universal rectangular dual for  $\mathcal{H}_1$ while the second possibility shows that on deleting $R_t$ from $F$, the remaining part is not a rectangular dual (see Fig. \ref{fig:f4}c) for  $\mathcal{H}_1$, but can be transformed to an area universal rectangular dual by adjusting the dimension of the rectangle $R_p$ as shown in \ref{fig:f4}d.
\par
 Proceeding as before, we can find an exterior vertex $v_k \in S-v_t$ since $\mathcal{H}_1$ being an induced subgraph of $\mathcal{G}$  admits an area-universal rectangular dual and the induced subgraph  $\mathcal{H}_2$ of $\mathcal{H}_1$ obtained by deleting  $v_k$ from $\mathcal{H}_1$ admits an area-universal rectangular dual.  Continuing  in this way until $S= \phi$, we find that $\mathcal{H}_k$ admits an area-universal rectangular dual. Hence the proof.
\end{proof}

\section{Concluding Remarks and Future Task} \label{sec3}

In this paper, we have identified a class $\mathcal{C}$ of RDGs   in which each RDG can be realized by  an area-universal rectangular dual. Eppstein et al.\cite{Eppstein12} described an algorithm to construct an area-universal rectangular dual for an RDG if it exists. In general, this algorithm is not fully polynomial. We have described  a polynomial time algorithm to construct  an area-universal rectangular dual  for an RDG that belongs to $\mathcal{C}$. Further each graph of $\mathcal{C}$ has been characterized by the fact that every induced rectangularly dualizable subgraph of each of its graph admits an area-universal rectangular dual. 
\par
Recently the most challenging problem  is  to find an efficient algorithm to construct  an area-universal rectangular dual for a given RDG, if it exists. 
\par
It can be observed that the importance of class $\mathcal{C}$ lies in the fact that  if a given RDG $\mathcal{G}$ can be split into a finite union $U$ of RDGs of $\mathcal{C}$, then $\mathcal{G}$ also admits an area-universal rectangular dual. In fact, the corresponding area-universal rectangular duals to each of the RDGs belonging to  $U$ can be glued to obtain an  area-universal rectangular dual for $\mathcal{G}$. To the contrary, it is interesting to find out in future that if a given RDG can not be split into a finite union $U$ of RDGs of $\mathcal{C}$, then whether  this RDG can be realized by  an area-universal rectangular dual.

 \bibliographystyle{elsarticle-num} 
\bibliography{references}

\end{document}